\documentclass{amsart}
\usepackage[frenchb, english]{babel}
\usepackage[utf8]{inputenc}

\usepackage{amsthm,amsmath,amssymb,amstext,amsfonts,bm}

\usepackage{array}   

\usepackage{caption}
\usepackage{graphicx}

\usepackage{todonotes}
\usepackage{hyperref}
\usepackage[capitalize]{cleveref}
\newcommand{\field}[1]{\mathbb{#1}}

\newcommand{\N}{\field{N}}

\def\O{\mathcal{O}}

\def\la{\lambda}
\def\si{\sigma}

\def\eps{\varepsilon}

\numberwithin{equation}{section}
\newtheorem{theorem}{Theorem}
\newtheorem{lemma}[theorem]{Lemma}
\newtheorem{corollary}[theorem]{Corollary}
\newtheorem{conjecture}[theorem]{Conjecture}

\newtheorem{proposition}[theorem]{Proposition}
\theoremstyle{remark}
\newtheorem{remark}[theorem]{Remark}

\renewenvironment{proof}[1][Proof]{\begin{trivlist}
\item[\hskip \labelsep {\bfseries #1:}]}{\qed\end{trivlist}}

\title[Asymptotics for skew standard Young tableaux]{Asymptotics for skew standard Young tableaux via bounds for characters}
\author{Jehanne Dousse and Valentin F\'eray}
\address{Institut f\"ur Mathematik, Universit\"at Z\"urich, Winterthurerstr. 190, CH-8032 Z\"urich, Switzerland}
\email{jehanne.dousse@math.uzh.ch,\ valentin.feray@math.uzh.ch}

\thanks{Both authors are partially supported by grant SNF-149461 from the Swiss National Science Foundation}

\begin{document}

\begin{abstract}
  We are interested in the asymptotics of the number of standard Young tableaux $f^{\la/\mu}$ of a given skew shape $\la/\mu$.
  We mainly restrict ourselves to the case where both diagrams are balanced,
  but investigate all growth regimes of $|\mu|$ compared to $|\la|$, from $|\mu|$ fixed to $|\mu|$ of order $|\la|$.
  When $|\mu|=o(|\la|^{1/3})$, we get an asymptotic expansion to any order.
  When $|\mu|=o(|\la|^{1/2})$, we get a sharp upper bound.
  For bigger $|\mu|$, we prove a weaker bound and give a conjecture on what we believe to be the correct order of magnitude.
  
  Our results are obtained by expressing $f^{\la/\mu}$ in terms of irreducible character values of the symmetric group
  and applying known upper bounds on characters.
\end{abstract}

\maketitle

\section{Introduction and statement of results}
\subsection*{Background}
Standard Young tableaux of a given shape $\lambda$ are standard combinatorial objects
coming from the representation theory of symmetric groups and the theory of symmetric functions;
we refer the reader to \cite{AR15} for a recent survey on the topic.
The number $f^\la$ of such tableaux is given by the well-known hook-length formula of Frame,
Robinson and Thrall \cite{HookLengthFormula1954}.
This exact product formula is also suited for asymptotic analysis:
for example, if $\la$ has fewer than $L\sqrt{|\la|}$ rows and columns (for some constant $L$),
then we have 
\begin{equation}
  \log f^\la=\tfrac12 |\la| \log |\la| +\O(|\la|).
  \label{eq:asymp_fla}
\end{equation}
This formula (with a precise version of the $\O(|\la|)$) is a key ingredient
to find the limit shape of random Young diagrams distributed according to the Plancherel measure;
see \cite[Chapter 1]{RomikLIS} for an introduction to this wide subject.
\medskip

A natural generalization of the problem is to consider skew shapes $\la/\mu$,
that is the diagram obtained by removing a smaller diagram $\mu$
from the top-left corner of a bigger diagram $\la$.
The number of standard Young tableaux of shape $\la/\mu$ is usually denoted $f^{\la/\mu}$
and is an object of interest in algebraic combinatorics.
In general, there is no product formula for $f^{\la/\mu}$, but several papers have been devoted to finding special shapes for which 
product formulas hold.  The most recent result in this direction is a formula for a six-parameter family of skew shapes given by
Morales, Pak and Panova \cite{MPP3}, which generalizes previous results of Kim and Oh \cite{KimOh} and DeWitt \cite{DeWitt}.
Another interesting problem is the asymptotic analysis of $f^{\la/\mu}$,
which we shall discuss in this paper.
\medskip

Even in the simplest case of a (non-skew) shape $\la$, the asymptotics of $f^\la$ largely depend on how the shape $\la$ tends to infinity: 
does it have rows and/or columns of linear size (this is sometimes referred to
as the Thoma-Vershik-Kerov regime) or, at the opposite, are the numbers of rows and columns
of the same order as the square-root of the size?
Since the numbers $f^{\la/\mu}$ depend on two partitions, there are even more possible asymptotic regimes.
In this article, to simplify the discussion,
we mainly focus on the case where {\bf both diagrams $\la$ and $\mu$ are balanced}; 
i.e., with the exception of \cref{th: general1,th: k general2}, 
we assume throughout the paper that
there exists $L$ such that $\la$ (resp. $\mu$)
has less than $L\sqrt{|\la|}$ (resp. $L\sqrt{|\mu|}$) rows and columns.
$L$ should be considered fixed and all constants, 
including the ones in the $\O$ and $o$ symbols, might depend on $L$.
However, the methods developed here can be used in more generality,
as seen in \cref{th: general1,th: k general2} below.
\medskip

We start the discussion with the case where $\mu$ is a fixed partition.
It has been proved that
for balanced diagrams $\la$
we have
\begin{equation}
  A_{\lambda/ \mu} := |\mu|! \frac{f^{\lambda/\mu}}{f^{\lambda}f^{\mu}} \underset{|\la| \to \infty}{\sim} 1. 
  \label{eq:equiv_stanley}
\end{equation}
(In fact, we only need to assume the number of rows and columns to be sublinear here.)
This result was first discovered by the Russian school and appeared in Kerov's unpublished doctoral thesis;
it can also be easily inferred from results of Olshanski and Okounkov on shifted Schur functions \cite{OkOl1998}.
It was then independently rediscovered by Stanley \cite{Stanley}.
Finally, Corteel, Goupil and Schaeffer \cite{CorteelGoupilSchaeffer2004} gave an alternative proof of it
via content evaluation of symmetric functions.
All these works use (shifted) symmetric function manipulations and asymptotic results
for characters or symmetric function evaluations.

On the other side of the spectrum, Morales, Pak, Panova and Tassy  \cite{PakMP,MPTassy} investigated
various situations where the size $k$ of $\mu$ grows linearly with the size $n$ of $\la$.
In all cases with balanced diagrams, they obtained
\begin{equation}
 \log  f^{\lambda/\mu} = \tfrac{1}{2} |\la/\mu| \log |\la/\mu| +\O(|\la/\mu|),
   \label{eq:est_MPP}
 \end{equation}
and gave precise estimates for the $\O(|\la/\mu|)$ term.
In this framework, again, it seems that $A_{\lambda/ \mu} = |\mu|! \tfrac{f^{\lambda/\mu}}{f^{\lambda}f^{\mu}}$
is a meaningful normalization since $\log(A_{\lambda/ \mu})=\O(|\la/\mu|)$ while
all factors in the definition of $A_{\lambda/ \mu}$ are significantly bigger.
\medskip

The goal of this paper is to investigate the behaviour of $A_{\lambda/ \mu}$
in intermediate regimes, that is when $1 \ll |\mu|\ll |\la|$.
We get various results, depending on the growth of $k:=|\mu|$
compared to $n:=|\la|$.
\medskip

\subsection*{A probabilistic interpretation of $A_{\lambda/\mu}$}
Before stating our result, let us present a motivation 
for studying $f_{\la/\mu}$ asymptotically through the quotient $A_{\lambda/\mu}$.
For a partition $\la$, we consider a uniform random standard tableau $T_\lambda$ of shape $\lambda$.
Large random standard tableaux have been studied in the probabilistic literature,
see for instance \cite{mckay02tableaux,olshanski01tableaux,pittel07tableaux,sun18tableaux}
or \cite{biane98free,sniady06gaussian} for the equivalent problem of studying random irreducible components
of restrictions of $S_n$-representations.
We also mention that Stanley's motivation for proving \eqref{eq:equiv_stanley} was to re-prove and improve
some results on random tableaux from \cite{mckay02tableaux}.

One way to study such tableaux is to consider their {\em level sets}, i.e., given $k$,
the diagram $L^{(k)}(T_\lambda)$ consisting of boxes with entries at most $k$ in $T_\lambda$.
With basic combinatorics, we see that, for given $\mu$ and $\lambda$, we have
\[ \mathbb P\big( L^{(k)}(T_\lambda)=\mu \big) = 
\frac{f^\mu\, f^{\la/\mu}}{f^\la} = A_{\lambda/ \mu} \frac{(f^\mu)^2}{|\mu|!}. \]
Therefore, the quantity $A_{\lambda/ \mu}$ is a correction factor
between the distribution of $L^{(k)}(T_\lambda)$
and the well-studied Plancherel measure on Young diagrams of size $k$.
With this viewpoint, \cref{eq:equiv_stanley} means that, 
for fixed $k$, the diagram $L^{(k)}(T_\lambda)$ is asymptotically Plancherel distributed.
We extend this result below to the range $k=o \left(n^{1/3}\right)$ 
and analyse the correction factor $A_{\lambda/ \mu}$ for larger values of $k$.
\medskip

\subsection*{Results}
When $k= o \left(n^{1/3}\right)$, we get an asymptotic expansion of $A_{\lambda/ \mu}$
to any order.
This extends the results of Kerov-Stanley for fixed $k$
(see the discussion after Theorem 3.2 in \cite{Stanley}).
The terms in this expansion involve characters of the symmetric group,
so we first need to introduce some terminology.
For a permutation $\sigma$,
we denote by $\ell_T(\sigma)$ its {\em absolute length}, i.e the minimal number of transpositions (not necessarily adjacent)
needed to factorize $\sigma$.
Also, $\chi^\lambda(\sigma)$ is the character of the irreducible symmetric group representation
associated with $\lambda$ evaluated on $\sigma$.
(If $\lambda$ has size $n$ and $\sigma$ is a permutation in the symmetric group $S_k$ with $k <n$,
we implicitly use the injection $S_k \subset S_n$ consisting in fixing integers $j>k$.)
\begin{theorem}
\label{th: k small}
Let  $\lambda \vdash n$ and $\mu \vdash k$ be balanced, with 
$k = o\left(n^{1/3 } \right)$. Then for any natural integer $r$ (not depending on $k$ and $n$), we have as $n$ tends to infinity,
$$A_{\lambda/ \mu} = \sum_{ \substack{\sigma \in S_k,\\ \ell_T( \sigma) \leq r}} \frac{\chi^{\lambda}(\sigma)}{f^{\lambda}} \frac{\chi^{\mu}(\sigma)}{f^{\mu}} + \O\left( \left( k^{\frac{3}{2}}n^{- \frac{1}{2}} \right)^{r+1} \right).$$
\end{theorem}
Characters on permutations of short absolute lengths have explicit expressions,
which allow to give character-free versions of the above asymptotic estimate for small values of $r$.
For $r=0$, the only permutation of absolute length $0$ being the identity, we get the following extension
of Kerov-Stanley's result.
\begin{corollary}
\label{cor: k<1/3}
Let  $\lambda \vdash n$ and $\mu \vdash k$ be balanced, with 
$k = o\left(n^{1/3 } \right)$. Then we have
$$A_{\lambda/ \mu} = 1 + \O\left(k^{\frac{3}{2}}n^{- \frac{1}{2}}\right).$$
In other words,
$$f^{\lambda/ \mu} = \frac{f^{\lambda}f^{\mu}}{k!} \left(1+\O\left(k^{\frac{3}{2}}n^{- \frac{1}{2}}\right) \right) .$$
\end{corollary}\smallskip

To state the case $r=1$, we denote $b(\la)=\sum_{i \ge 1} (i-1)\la_i$ and recall that,
for any transposition $\tau$, we have
\[\frac{\chi^{\lambda}(\tau)}{f^{\lambda}}=\frac{2}{n(n-1)}(b(\la')-b(\la)).\]
See, {\em e.g.}, \cite[Section I.7, example 7]{McDo}; here, $\la'$ is as usual the conjugate of $\la$.
The permutations of absolute length at most $1$ in $S_k$ are the identity and $\frac{k(k-1)}{2}$
transpositions, so that we have the following.
\begin{corollary}
\label{cor: k<1/3'}
Let  $\lambda \vdash n$ and $\mu \vdash k$ be balanced, with 
$k = o\left(n^{1/3 } \right)$. Then we have
$$A_{\lambda/ \mu} = 1 + \frac{2}{n(n-1)} (b(\la')-b(\la))(b(\mu')-b(\mu))+\O\left(k^3 n^{-1}\right).$$
\end{corollary}

\medskip

We now discuss what happens for larger values of $k$.
When $k$ is at most of order $n^{1/2}$, we can prove the following upper bound.
\begin{theorem}
\label{th: k medium}
There exist positive constants $C_1$ and $C_2$ such that the following holds.
Let  $\lambda \vdash n$ and $\mu \vdash k$ be balanced, with $k < C_1 n^{1/2}$. 
Then we have
$$A_{\lambda/ \mu} \leq e^{C_2 k^{\frac{3}{2}} n^{- \frac{1}{2}}}.$$
\end{theorem}

The bound of Theorem \ref{th: k medium} is in some sense tight; in \cref{sec:numerics}, we give families of skew shapes for which $\log A_{\lambda/ \mu} = \Theta \left(|\mu|^{\frac{3}{2}}|\la|^{-\frac{1}{2}}\right)$ (here $f = \Theta (g)$ means that there exists a constant $C>0$ such that $f=Cg +o(g)$).
These skew shapes have been chosen so that
$f^{\la/\mu}$ (and hence $A_{\la/\mu}$) admits a product formula
and is therefore easily computable.
Even if this product formula is explicit, the derivation of its asymptotics is cumbersome 
and was done on a computer.
\medskip

We next investigate the case where $k \geq C_1 n^{1/2}$.
In this case, we can only prove the following upper bound.
\begin{theorem}
\label{th: k big}
Let  $\lambda \vdash n$ and $\mu \vdash k$ be balanced, with $k \geq C_1 n^{1/2}$.
Then we have
$$A_{\lambda/ \mu} \leq e^{ k\left( \log \frac{k^2}{n} + \O(1) \right)}.$$
\end{theorem}
Note that when $k$ is of order $n^{1/2}$, the upper bound in \cref{th: k medium}
is $e^{\Theta(n^{1/4})}$, while the one in \cref{th: k medium} is $e^{\Theta(n^{1/2})}$,
raising the question of the existence of a sudden change in the asymptotic behavior of $A_{\lambda/ \mu}$.
We believe that such a change does not exist,
but that this is rather an artifact of our method.
In fact, we conjecture that the bound given in \cref{th: k medium}
holds without hypothesis on $k$:

\begin{conjecture}
Let  $\lambda \vdash n$ and $\mu \vdash k$ be balanced. Then we have
$$e^{-C k^{\frac{3}{2}} n^{- \frac{1}{2}}} \leq A_{\lambda/ \mu} \leq e^{C k^{\frac{3}{2}} n^{- \frac{1}{2}}},$$
for some positive constant $C$. 
\label{conj}
\end{conjecture}
Note that this conjecture, unlike our previous two theorems, also includes a lower bound.
The conjecture is supported by numerical evidence obtained as above:
we computed the asymptotic expansion  of $\log A_{\la/\mu}$,
in various cases with product formulas.
These computations are also presented in \cref{sec:numerics}.

We believe that the techniques developed by Morales, Pak, Panova and Tassy \cite{PakMP,MPTassy}
can be used to prove the conjecture when $k$ is of order $n$. 
We do not know however how to attack it in the regime $n^{1/2} \ll k \ll n$,
or how to prove the lower bound for $k \ll n^{1/2}$.
The approach of this paper is not suited for lower bounds,
and new ideas will be needed to improve the upper bound
in the regime $n^{1/2} \ll k \ll n$, see discussion below.

\medskip

In our last set of results, we relax the condition that the Young diagrams need to be balanced.
Let us consider more generally diagrams $\lambda \vdash n$ and $\mu \vdash k$ whose number of rows and columns are bounded by some functions $\alpha(n)$ and $\beta(k)$, respectively.
Note that $\alpha(n)$ (resp. $\beta(k)$) are always between $n^{1/2}$ and $n$ (resp. $k^{1/2}$ and $k$).
 \cref{th: k small,th: k medium}, respectively, are special cases of the following theorems.

\begin{theorem}
\label{th: general1}
Let  $\lambda \vdash n$ (resp. $\mu \vdash k$) with number of rows and columns at most $\alpha(n)$ (resp. $\beta(k)$), such
$k \beta(k) = o\left(n/\alpha(n) \right)$. Then for any natural integer $r$ (not depending on $k$ and $n$), we have as $n$ tends to infinity,
$$A_{\lambda/ \mu} = \sum_{ \substack{\sigma \in S_k,\\ \ell_T( \sigma) \leq r}} \frac{\chi^{\lambda}(\sigma)}{f^{\lambda}} \frac{\chi^{\mu}(\sigma)}{f^{\mu}} + \O\left( \left( k\beta(k)n^{-1} \alpha(n) \right)^{r+1} \right).$$
\end{theorem}

\begin{theorem}
\label{th: k general2}
There exist constants $C'_1$ and $C'_2$ such that the following holds.
Let  $\lambda \vdash n$ (resp. $\mu \vdash k$) with number of rows and columns at most $\alpha(n)$ (resp. $\beta(k)$), such that $k < C'_1 n/\alpha(n)$. 
Then we have
$$A_{\lambda/ \mu} \leq e^{C'_2 k\beta(k)n^{-1} \alpha(n)}.$$
\end{theorem}

\begin{remark}
All our constants $C_1, C_2, C'_1, C'_2$ in the above theorems have precise expressions depending on the constant $a$ appearing in the bound on characters from \cref{th:valpiotr}. These expressions can be found in the proofs of the corresponding theorems. 
Unfortunately, the constant $a$ in \cref{th:valpiotr} was not made explicit in the original paper \cite{ValentinPiotr},
so that we cannot make the constants $C_1, C_2, C'_1, C'_2$ completely explicit.
\end{remark}


\subsection*{Method}
We finish this introduction by a short discussion on the method used to prove our results.
As Stanley, we start with an expression of $f^{\la/\mu}$ (or equivalently $A_{\la/\mu}$)
as a sum involving irreducible characters of the symmetric group (see \cref{lem:stanley} below).
We then control the sum thanks to an upper bound on these characters
due to the second author and \'Sniady \cite{ValentinPiotr}.

In Section \ref{sec:limits}, we argue that using other bounds for characters known to date 
would not allow us to improve the bound in Theorem~\ref{th: k big}.
Namely, in the sum expressing $A_{\la/\mu}$ in terms of characters,
we replace each (normalized) character $\chi^\la(\si)/f^{\la}$ (or $\chi^\mu(\si)/f^{\mu}$)
by the best bound from the literature (at least the best bound we are aware),
giving a natural upper bound $B_{\la/\mu}$ for $A_{\la/\mu}$;
note that the best bound is chosen independently for each summand, i.e.
we are using different bounds for different summands, always the best possible.
We prove that the resulting bound $B_{\la/\mu}$ is of order at least 
$e^{k\left( \log \frac{k^2}{n} + \O(1) \right)}$, i.e. not better that what we
have in Theorem~\ref{th: k big}.

In comparison, the method of Morales, Pak and Panova \cite{PakMP} is completely different,
being based on a (non-multiplicative) hook-length formula for skew diagrams.

\section{Proofs of the asymptotic estimates}
Let us consider two Young diagrams $\lambda \vdash n$ and $\mu \vdash k$, and denote by $r(\lambda)$ (resp. $c(\lambda)$) the number of rows (resp. colums) of $\lambda$ (idem for $\mu$). We assume that
\[
r(\lambda), c(\lambda) \leq \alpha(n),\quad
r(\mu), c(\mu) \leq \beta(k),
\]
for some functions $\alpha$ and $\beta$. Note that, by the definition of Young diagrams, we necessarily have
\[
n^{1/2} \leq \alpha(n) \leq n ,\quad
k^{1/2} \leq \beta(k)\leq k.
\]
 
\subsection{Preliminaries}
We start with a lemma expressing $f^{\la/\mu}$ (or equivalently $A_{\la,\mu}$)
as a sum of irreducible character values of the symmetric group.
This formula already appears in a slightly different form in Stanley
\cite[Theorem 3.1]{Stanley},
but here we give a more direct proof.
\begin{lemma}[Stanley]
\label{lem:stanley}
Let $\mu$ be a partition of $k$, and let $n \geq k$. Then for all partitions $\lambda$ of $n$, we have
$$A_{\lambda/\mu} =  \sum_{\sigma \in S_k} \frac{\chi^{\lambda}(\sigma)}{f^{\lambda}} \frac{\chi^{\mu}(\sigma)}{f^{\mu}}.$$
\end{lemma}
\begin{proof}
  Iterating the branching rule for representations of the symmetric group (see, e.g., \cite[Section 2.8]{Sagan}),
  we get that, if $\sigma \in S_k$,
  \[\chi^{\lambda}(\sigma)=\sum_{\mu \vdash k} f^{\la/\mu}\, \chi^{\mu}(\sigma),\]
  where the sum runs over partitions of $k$.
  Since $(\sigma \mapsto \chi^{\mu}(\sigma))_{\mu \vdash k}$ forms an orthogonal basis of central functions on $S_k$
  \cite[Theorem 1.9.3]{Sagan}, the coefficient $f^{\la/\mu}$ of $\chi^{\mu}(\sigma)$ 
  in the expansion of $\chi^{\lambda}(\sigma)$ can be obtained by a scalar product computation:
  \[f^{\la/\mu}= \langle \chi^{\lambda}(\sigma),\chi^{\mu}(\sigma) \rangle
  =\frac{1}{k!} \sum_{\sigma \in S_k} \chi^{\lambda}(\sigma) \chi^{\mu}(\sigma).\]
  Dividing by $f^{\lambda} \, f^{\mu}$ and multiplying by $k!$ gives the result.
\end{proof}

Our results are also based on asymptotic bounds for characters 
due to the second author and \'Sniady \cite{ValentinPiotr}.

\begin{theorem}[F\'eray-\'Sniady]
\label{th:valpiotr}
There exists a constant $a>1$ such that for every partition $\nu \vdash m$ and every permutation $\sigma \in S_m$,
$$   \left| \frac{\chi^{\nu}(\sigma)}{f^{\nu}} \right| \leq \left[a \max \left(\frac{r(\nu)}{m},\frac{c(\nu)}{m}, \frac{\ell_T(\sigma )}{m} \right) \right]^{\ell_T(\sigma)}.$$
\end{theorem}

With a given bound on $r(\nu)$ and $c(\nu)$, this result specializes as follows.
\begin{corollary}
\label{lem:boundchar}
Let $\nu \vdash m$ be a partition with
\begin{align*}
r(\nu), c(\nu) &\leq \gamma(m),
\end{align*}
with $m^{1/2} \leq \gamma(m) \leq m$,
and let $\sigma$ be a permutation in $S_m$.
Then the following holds.
\begin{itemize}
\item When $\ell_T(\sigma) \leq \gamma(m)$,
\begin{equation}
\label{eq:1}
\left| \frac{\chi^{\nu}(\sigma)}{f^{\nu}} \right| \leq \left(\frac{a\gamma(m)}{m} \right)^{\ell_T(\sigma)}.
\end{equation}
\item When $\ell_T(\sigma) > \gamma(m)$,
\begin{equation}
\label{eq:4}
\left| \frac{\chi^{\nu}(\sigma)}{f^{\nu}} \right| \leq \left(\frac{a \ell_T(\sigma)}{m} \right)^{\ell_T(\sigma)}.
\end{equation}
\end{itemize}
\end{corollary}

\subsection{Proof of Theorem \ref{th: general1}}
Let us now prove Theorem~\ref{th: general1}. We assume $k \beta(k) = o\left(n/\alpha(n) \right)$.
By Lemma \ref{lem:stanley}, we have
\[
A_{\lambda/\mu} =  \sum_{\sigma \in S_k} \frac{\chi^{\lambda}(\sigma)}{f^{\lambda}} \frac{\chi^{\mu}(\sigma)}{f^{\mu}}
= \sum_{i=0}^k \sum_{ \substack{\sigma \in S_k,\\ \ell_T(\sigma) = i}} \frac{\chi^{\lambda}(\sigma)}{f^{\lambda}} \frac{\chi^{\mu}(\sigma)}{f^{\mu}}
\]

Let $r$ be a fixed positive integer. We may assume $r \leq \beta(k)$; otherwise, simply replace $\beta(k)$ by $\max (\beta(k),r)$.  We split this sum into three parts:
\begin{equation}
\label{eq:decoupage k small}
A_{\lambda/\mu} = \sum_{i=0}^r \sum_{ \substack{\sigma \in S_k,\\ \ell_T(\sigma) = i}} \frac{\chi^{\lambda}(\sigma)}{f^{\lambda}} \frac{\chi^{\mu}(\sigma)}{f^{\mu}} + S_1 + S_2,
\end{equation}
where
\begin{align}
S_1&:= \sum_{i=r+1}^{\beta(k)} \sum_{ \substack{\sigma \in S_k,\\\ell_T(\sigma)= i}} \frac{\chi^{\lambda}(\sigma)}{f^{\lambda}} \frac{\chi^{\mu}(\sigma)}{f^{\mu}},\\
S_2&:= \sum_{i=\beta(k)+1}^k \sum_{ \substack{\sigma \in S_k,\\ \ell_T(\sigma) = i}} \frac{\chi^{\lambda}(\sigma)}{f^{\lambda}} \frac{\chi^{\mu}(\sigma)}{f^{\mu}}.
\label{eq:S2}
\end{align}

We now wish to bound $S_1$ and $S_2$. To do so, we use a lemma from \cite{ValentinPiotr}.
\begin{lemma}[Lemma 14 from \cite{ValentinPiotr}]
\label{lem:nbperm}
For all $k,i \in \N$, we have
$$\#\left\{\sigma \in S_k : \ell_T(\sigma)=i\right\} \leq \frac{k^{2i}}{i!}.$$
\end{lemma}

Using Lemma \ref{lem:nbperm} and Equation \eqref{eq:1} for $\nu=\lambda$ and $\nu=\mu$, we obtain
$$|S_1| \leq \sum_{i=r+1}^{\beta(k)} \frac{k^{2i}}{i!} \left(\frac{a\alpha(n)}{n} \right)^{i} \left(\frac{a\beta(k)}{k} \right)^{i}.$$
We can now bound $S_1$ by the tail of an exponential series:
$$|S_1| \leq \sum_{i=r+1}^{+ \infty} \frac{ \left(a^2k\beta(k)n^{-1}\alpha(n)\right)^i}{i!}.$$
As $k \beta(k) = o\left(n/\alpha(n) \right)$, we conclude that
\begin{equation}
\label{eq:boundS1}
|S_1| =  \O\left(  \left(k\beta(k)n^{-1}\alpha(n)\right)^{r+1} \right).
\end{equation}

Let us now turn to $S_2$.
Since $\ell_T(\sigma) =i > \beta(k)$ in the summation index of $S_2$,
we must now apply \eqref{eq:4} to $\chi^\mu(\sigma)$.
On the other hand, since $\alpha(n)$ (resp. $\beta(k)$) is at least $n^{1/2}$ (resp. $k^{1/2}$),
we have, for $n$ big enough,
\[\ell_T(\sigma) \le k \le (k \beta(k))^{2/3} = o\left(\left(n/\alpha(n) \right)\right)^{2/3} \leq \alpha(n).\]
As a consequence we can still apply \eqref{eq:1} to $\chi^\la(\sigma)$.
Combining these bounds with Lemma \ref{lem:nbperm}, we have:
$$ |S_2| \leq \sum_{i=\beta(k)+1}^k  \frac{k^{2i}}{i!} \left(\frac{a\alpha(n)}{n} \right)^{i} \left(\frac{a i}{k} \right)^{i}.$$
Using 
$i! \geq \tfrac{i^i}{e^i},$
we obtain
$$ |S_2| \leq \sum_{i=\beta(k)+1}^k  \left(a^2e k n^{-1}\alpha(n) \right)^{i}.$$
This geometric series converges for $k < \left(a^2e \right)^{-1} n/\alpha(n)$, and thus $|S_2|$ is bounded by
\begin{equation}
\label{eq:boundS2}
|S_2| \leq \left(a^2e k n^{-1}\alpha(n)\right)^{\beta(k)+1} \frac{1}{1-\left(a^2e k n^{-1}\alpha(n) \right)}.
\end{equation}
When $k \beta(k) = o\left(n/\alpha(n) \right)$ and $\beta(k)\geq r$, the upper bound in \eqref{eq:boundS2} is negligible compared to \eqref{eq:boundS1} since the exponent $\beta(k)+1$ is at least $r+1$.
Thus combining Equations \eqref{eq:decoupage k small}, \eqref{eq:boundS1} and \eqref{eq:boundS2}, we obtain
$$A_{\lambda/\mu} = \sum_{ \substack{\sigma \in S_k,\\ \ell_T(\sigma) \leq r}} \frac{\chi^{\lambda}(\sigma)}{f^{\lambda}} \frac{\chi^{\mu}(\sigma)}{f^{\mu}} +  \O\left(  \left(k\beta(k)n^{-1}\alpha(n)\right)^{r+1} \right).$$
Theorem \ref{th: general1} is proved. \qed

\begin{proof}[Proof of Theorem \ref{th: k small}]
  Theorem \ref{th: k small} is obtained by specializing Theorem \ref{th: general1} to the case \hbox{$\alpha(n)=L \sqrt{n}$} and $\beta(k)=L \sqrt{k}$.
\end{proof}


\subsection{Proof of Theorem \ref{th: k general2}}
\label{sec:proof k medium}
Let us assume that $k < C'_1 n/\alpha(n)$, with $C'_1 :=  \left(a^2e \right)^{-1}$.
We proceed similarly to Theorem \ref{th: general1}. Let us write
$$A_{\lambda/\mu} = S'_1 +S_2,$$
where $S_2$ is defined in \eqref{eq:S2} and
\begin{equation}
  S'_1:= \sum_{i=0}^{\beta(k)} \sum_{ \substack{\sigma \in S_k,\\ \ell_T(\sigma) = i}} \frac{\chi^{\lambda}(\sigma)}{f^{\lambda}} \frac{\chi^{\mu}(\sigma)}{f^{\mu}}.
  \label{eq:Sp1}
\end{equation}
By the same argument as in the proof of Theorem \ref{th: k small}, we can bound this by an exponential sum and obtain
\begin{equation}
\label{eq:S'1}
|S'_1| \leq e^{a^2k\beta(k)n^{-1}\alpha(n)}.
\end{equation}
Thanks to our choice for $C'_1$, the sum $S_2$ is still a truncated convergent geometric series and \eqref{eq:boundS2} still holds. Combining this with \eqref{eq:S'1} completes the proof. \qed


\begin{proof}[Proof of Theorem \ref{th: k medium}]
Theorem \ref{th: k medium} is obtained by specializing Theorem \ref{th: k general2} to the case $\alpha(n)=L \sqrt{n}$ and $\beta(k)=L \sqrt{k}$. In that case the constant $C'_1$ becomes $C_1 =\left(a^2Le \right)^{-1}$. For the constant $C_2$, any value $C_2>a^2 L^2$ works for $n$ and $k$ large enough.
However the authors of \cite{ValentinPiotr} do not give estimates on the value of $a$ in \cref{th:valpiotr},
so this is not useful in practice to compute $C_1$ or $C_2$.
\end{proof}

\subsection{Proof of Theorem \ref{th: k big}}
From now on we assume that $\lambda$ and $\mu$ are balanced, i.e.
\[
r(\lambda), c(\lambda) \leq L \sqrt{n},\quad
r(\mu), c(\mu) \leq L \sqrt{k},
\]
for some positive constant $L$. 

We prove the bound for large $k$, i.e. $k \ge (a^2Le)^{-1} \sqrt{n}$. 
Again, we split $A_{\lambda/\mu}$ into several parts:
$$A_{\lambda/\mu} = S'_1 +S'_2+S'_3,$$
where $S'_1$ was defined in \eqref{eq:Sp1} for $\beta(k) = L \sqrt{k}$,
$$S'_2 := \sum_{i=L \sqrt{k}+1}^{L \sqrt{n}} \sum_{ \substack{\sigma \in S_k,\\ \ell_T(\sigma) = i}} \frac{\chi^{\lambda}(\sigma)}{f^{\lambda}} \frac{\chi^{\mu}(\sigma)}{f^{\mu}},$$
and
$$S'_3 := \sum_{i=L\sqrt{n}+1}^{k} \sum_{ \substack{\sigma \in S_k,\\ \ell_T(\sigma) = i}} \frac{\chi^{\lambda}(\sigma)}{f^{\lambda}} \frac{\chi^{\mu}(\sigma)}{f^{\mu}}.$$
Note that it might happen that $k \le L\sqrt{n}$. 
In this case, $S'_3=0$ and all terms for $i>k$ in $S'_2$ are $0$ as well,
but this will not affect the bounds below.
\smallskip

Consider first $S'_1$. Since it is a sum over permutations $\sigma$ with $\ell_T(\sigma) \le L\sqrt{k} \le L\sqrt{n}$,
we can still apply \eqref{eq:1} and
 Equation \eqref{eq:S'1} still holds.
 \smallskip

 For $S'_2$, we should apply \eqref{eq:1} to $\chi^\lambda(\sigma)$ and \eqref{eq:4} to $\chi^\mu(\sigma)$.
 Combining with Lemma \ref{lem:nbperm} and using again the bound $i!\geq \tfrac{i^i}{e^i}$, we have:
\begin{align*}
|S'_2| &\leq \sum_{i=L \sqrt{k}+1}^{L \sqrt{n}}  \frac{k^{2i}}{i!} \left(\frac{aL}{\sqrt{n}} \right)^{i} \left(\frac{a i}{k} \right)^{i}\\
&\leq \left(a^2Le k n^{-\frac{1}{2}} \right)^{L \sqrt{k}+1} \times \frac{\left(a^2Le k n^{-\frac{1}{2}} \right)^{{L \sqrt{n}}-L\sqrt{k}}-1}{\left(a^2Le k n^{-\frac{1}{2}} \right)-1}.\\
&\leq C_3 \left(a^2Le k n^{-\frac{1}{2}} \right)^{L \sqrt{n}},
\end{align*}
for some constant $C_3$. Note that this bound is different from \eqref{eq:boundS2},
since we now have a {\em divergent} geometric series.
Thus
\begin{equation}
\label{eq:boundS'2}
|S'_2| \leq e^{\sqrt{n} \left(\log (kn^{-\frac{1}{2}}) + \O(1) \right)}.
\end{equation}

Let us finally turn to $S'_3$. By Lemma \ref{lem:nbperm} and Equation \eqref{eq:4} 
applied to $\chi^\lambda(\sigma)$ and $\chi^\mu(\sigma)$, we have
$$|S'_3| \leq  \sum_{i=L \sqrt{n}+1}^k \frac{k^{2i}}{i!} \left(\frac{ai}{n} \right)^{i} \left(\frac{a i}{k} \right)^{i}.$$
As before we bound the first factor $i^i$ by $i!e^i$, and using $i \leq k$, we bound $\left(\frac{a i}{k} \right)^{i}$ by $a^i$. This gives
\begin{align*}
  |S'_3| &\leq  \sum_{i=L \sqrt{n}+1}^k \left(\frac{a^2k^{2}e}{n}\right)^i.\\
  & \le \left(a^2k^{2}e n^{-1}\right)^{L \sqrt{n}+1} \frac{\left(a^2k^{2}e n^{-1}\right)^{k-L\sqrt{n}}-1}{\left(a^2k^{2}e n^{-1}\right)-1}\\
  & \le C_4 \left(a^2k^{2}e n^{-1}\right)^k=e^{k \left(\log \frac{k^2}{n} + \O(1) \right)},
\end{align*}
for some constant $C_4$.

From \eqref{eq:S'1} and \eqref{eq:boundS'2}, we know that $S'_1$ and $S'_2$ are also bounded by $e^{k (\log \tfrac{k^2}{n} + \O(1))}$
(for \hbox{$k=\Theta(n^{1/2})$}, the upper bound given in \eqref{eq:boundS'2} is of this order; otherwise it is smaller).
This completes the proof of Theorem \ref{th: k big}. \qed

\section{Limitations of the method}
\label{sec:limits}
As mentioned in the introduction, we do not believe the bound in \cref{th: k big} to be sharp.
Nevertheless we argue in this section that, with the known bounds on characters and the method used in this paper,
we cannot improve \cref{th: k big}.
Such an improvement would require either new bounds on characters,
or taking signs and cancellations appearing in \cref{lem:stanley} into account,
or a completely different method.
Improving upper bounds on characters or having a good understanding of their signs
are in general believed to be hard problems,
so we think that an improvement is more likely to come from a totally different method.
\medskip

As far as we are aware of, the best bounds on characters available in the literature are the following.
\begin{itemize}
  \item In \cite{Roichman1996}, Roichman proved that there exist constants $b>0$ and $q \in (0,1)$
    such that for any partition $\nu$ and permutation $\sigma$ of the same size $m$, we have
    \begin{equation}
      \left|\frac{\chi^\nu(\sigma)}{f^\nu} \right| \le 
      \left[ \max \left( q,\tfrac{r(\nu)}{m},\tfrac{c(\nu)}{m} \right) \right]^{b\, \ell_T(\sigma)}.
      \label{eq:Roichmann}
    \end{equation}
    (In Roichman's paper, the exponent is in fact $b \, |\textrm{supp}(\sigma)|$,
    where $|\textrm{supp}(\sigma)|$ is the size of the support of $\sigma$,
    {\em i.e.} the number of its non-fixed points.
    This modification is however irrelevant
    since the value of $b$ is not known and, for all $\sigma$, we have $\ell_T(\sigma) \le \textrm{supp}(\sigma) \le 2\ell_T(\sigma)$.)
  \item In \cite{muller07fuchsian}, M\"uller and Schlage-Putcha 
    proved the following bound: for $m$ sufficiently large,
    for any partition $\nu$ of $m$ and permutation $\sigma$ of $m$ with $f$ fixed points, one has
    \begin{equation}
      \left|\frac{\chi^\nu(\sigma)}{f^\nu} \right| \le (f^\nu)^{-\delta(\si,\nu)},\text{ where }
    \delta(\si,\nu)=\left( (1-1/(\log m))^{-1}\, \frac{12 \log(m)}{\log(m/f)} +18 \right)^{-1}
    \label{eq:Bound_MSP}
 \end{equation}
  \item Shortly after this, Larsen and Shalev \cite{LarsenShalev2008}
    found a bound of the same kind with 
    a different exponent:
    for all $\eps >0$, there exists $N$ such that, for all integers $m>N$,
    partitions $\nu$ and permutations $\sigma$ both of size $m$, we have
    \begin{equation}
      \left|\frac{\chi^\nu(\sigma)}{f^\nu} \right| \le (f^\nu)^{B(\sigma)-1+\eps},
      \label{eq:LarsenShalev}
    \end{equation}
    where $B(\sigma)$ is some parameter between $0$ and $1$ associated with $\sigma$. Its complete definition
    is technical and irrelevant here, but we note that $B(\sigma)$ is small (i.e. far from one)
    if and only if $\sigma$ has few (i.e. a sublinear number of) fixed points.
  \item Lastly, we restate for the reader's convenience the bound of the second author and \'Sniady \cite{ValentinPiotr}:
there exists a constant $a>1$ such that for every partition $\nu \vdash m$ and every permutation $\sigma \in S_m$,
\begin{equation}
   \left| \frac{\chi^{\nu}(\sigma)}{f^{\nu}} \right| \leq \left[a \max \left(\frac{r(\nu)}{m},\frac{c(\nu)}{m}, \frac{\ell_T(\sigma)}{m} \right) \right]^{\ell_T(\sigma)}.
   \label{eq:ValentinPiotr}
 \end{equation}
\end{itemize}
The last bound is the best one if the two following conditions hold simultaneously:
the diagrams have no rows or columns of linear size (which we always assume here)
and the length of the permutation is sublinear (implying that the number of fixed point is of order $n$). 
In contrast, the third bound can be better for permutations $\sigma$ in $S_m$
with a sublinear number of fixed points.
Therefore, a natural idea is the following:
it might be possible to improve \cref{th: k big} by combining all these bounds,
i.e. take the minimal one for each permutation $\sigma$ (and diagram $\nu$).
We will see that it is not the case, and that improving \cref{th: k big} requires a new approach.

Call $U_{\text{R}}(\sigma,\nu)$ (resp. $U_{\text{MSP}}(\sigma,\nu)$, $U_{\text{LS}}(\sigma,\nu)$, $U_{\text{F\' S}}(\sigma,\nu)$)
the upper bound in \eqref{eq:Roichmann} (resp. \eqref{eq:Bound_MSP}, \eqref{eq:LarsenShalev} and \eqref{eq:ValentinPiotr}) and set
\[U_{\text{all}}(\sigma,\nu) = \min \big( U_{\text{R}}(\sigma,\nu), U_{\text{MSP}}(\sigma,\nu), U_{\text{LS}}(\sigma,\nu), U_{\text{F\' S}}(\sigma,\nu) \big).\]
From \cref{lem:stanley}, we know that
\[A_{\lambda/\mu} \le B_{\lambda/\mu} \text{, with }B_{\lambda/\mu}=\sum_{\sigma \in S_k} U_{\text{all}}(\sigma,\lambda) U_{\text{all}}(\sigma,\mu).\]
The main result of this section is a lower bound for $B_{\lambda/\mu}$,
which matches the upper bound in \cref{th: k big}.
This justifies our claim from the introduction, that using
other existing bounds on characters than the one from \cite{ValentinPiotr}
would not result in an improvement of \cref{th: k big}.
\begin{proposition}
Let  $\lambda \vdash n$ and $\mu \vdash k$ balanced, and assume $k \ge 2 L n^{\frac{1}{2}}.$ Then we have
$$B_{\lambda/ \mu} \geq e^{k\left( \log \frac{k^2}{n} + \O(1) \right)}.$$
\end{proposition}
\begin{proof}
  The idea of the proof consists in identifying in the sum defining $B_{\lambda/\mu}$
  a large subfamily of permutations $\sigma$ for which $U_{\text{all}}(\sigma,\lambda) U_{\text{all}}(\sigma,\mu)$
  is large and reasonably simple, so that the sum on this smaller set is already big enough.

  Fix a small number $\eta>0$, {\em e.g.}, $\eta=\frac18$. We denote by $F(\sigma)$ the number of fixed points of $\sigma$.
  We will consider, for each $k \ge 1$, the set
  \[ \Omega_k^\eta=\big\{\sigma \in S_k:\ F(\sigma) = \lfloor \eta\, k \rfloor 
  \text{ and } \ell_T(\sigma) \ge (1-\eta)^2k \big\}. \]
  In the following, all constants are independent from $k$ and $n$ (so a fortiori of $\lambda$ and $\mu$),
  but might depend on $\eta$.
  Our first claim is that there exists a constant $C_5$  such that
  \begin{equation}
    \#\Omega_k^\eta \ge (C_5)^k\, k^{(1-\eta)k}.
    \label{eq:sizeO}
  \end{equation}
  First consider the superset $\widetilde{\Omega}_k^\eta \supset \Omega_k^\eta$ 
  obtained by removing the condition on the length of $\ell_T(\sigma)$.
  Permutations in $\widetilde{\Omega}_k^\eta$ are uniquely obtained by choosing the $\lfloor \eta\, k \rfloor$ fixed points
  and choosing a derangement (i.e. a fixed-point free permutation) 
  of the remaining $\lceil (1-\eta)k \rceil$ points.
  Therefore
  \[\#\widetilde{\Omega}_k^\eta = \binom{k}{\eta k}\, D_{\lceil k(1-\eta)\rceil},\]
  where $D_K$ is the number of derangements of $K$ elements.
  It is well known that $D_K\sim \tfrac{K!}{e}$, so that a simple computation using Stirling's approximation gives
  \[\#\widetilde{\Omega}_k^\eta \ge \gamma_1^k\, k^{(1-\eta)k},\]
  for some constant $\gamma_1$ and $k$ large enough.

  It remains to see that $\Omega_k^\eta$ covers at least a  proportion $\gamma_2^k$ of $\widetilde{\Omega}_k^\eta$ (for some constant $\gamma_2$).
  In fact we will see that this proportion tends to $1$.
  Equivalently, we will prove that the proportion of derangements of $K$ elements
  with length at least $(1-\eta)K$ tends to 1.
  
  It is well-known since Feller \cite[p. 815]{Feller} 
  that the number of cycles in a uniform random permutation of size $K$
  is asymptotically normal with mean and variance $\log(K)$.
  In particular the proportion of permutations of size $K$ with less than $\eta K$ cycles
 - or equivalently with length at least $(1-\eta)K$ - tends to $1$.
  On the other hand the proportion of derangements tends to $1/e$.
  So the proportion of derangements with length at least $(1-\eta)K$
  among derangements should tend to $1$ as well.
  This completes the proof of \eqref{eq:sizeO}.

  In the following we assume that $\sigma$ is in $\Omega_k^\eta$.
  We now give a lower bound for the summand corresponding to such permutations
  in the definition of $B_{\lambda/\mu}$.
  It easy to check that for $k$ large enough, the parameter $B(\sigma)$
  appearing in the Larsen-Shalev bound will be arbitrary close to $1$ uniformly
  for $\sigma$ in $\Omega_k^\eta$ (since $\sigma$ has linearly many fixed points)
  so that the exponent $1-B(\sigma)+\eps$ in the Larsen-Shalev bound becomes eventually positive.
 The Larsen-Shalev bound does not give any information in this case (the RHS is bigger than $1$,
  while we know that the LHS is always smaller than 1).

  We now analyse M\"uller-Sclage Putcha bound.
  When $\sigma$ is in $\Omega_k^\eta$,
  some elementary analysis proves that $\delta(\si,\mu)=\O(\log(k)^{-1})$,
  while $\delta(\si,\la)=\O(kn^{-1}\log(n)^{-1})$
  (in the latter case, we see $\si$ as a permutation in $S_n$;
  it therefore has $n-k+\lfloor \eta k \rfloor$ fixed points).
  Using the trivial bound $f^\mu \le k!$ and $f^\la \le n!$,
  we get that both $U_{\text{MSP}}(\sigma,\mu)$ and $U_{\text{MSP}}(\sigma,\la)$
  are at least of order $e^{-\O(k)}$.
  Note that this is similar to Roichman's bound for such permutations.

  We can therefore forget about Larsen-Shalev and M\"uller-Schlage Putcha bounds
  when evaluating either $U_{\text{all}}(\sigma,\lambda)$ and $U_{\text{all}}(\sigma,\mu)$ for $\sigma$ is in $\Omega_k^\eta$.
  Besides, since the diagrams $\la$ and $\mu$ are balanced and
  since the length of $\sigma$ in $\Omega_k^\eta$ is always at least $(1-2\eta)k \ge L\sqrt{n}$
  (assuming $\eta <1/4$), the maximum in the other two bounds will never be reached by $\tfrac{r(\lambda)}{\sqrt{n}}$ or $\tfrac{c(\lambda)}{\sqrt{n}}$.
  We therefore have
  \begin{align*}
    U_{\text{all}}(\sigma,\lambda) &= \left[ \min\big(q^b,\tfrac{a\ell_T(\sigma)}{n}\big)\right]^{\ell_T(\sigma)} \ge 
    \left[ \min\big(q^b,\tfrac{a(1-\eta)^2 k}{n}\big)\right]^k \ge \big(C_6 \tfrac{k}{n}\big)^k,\\
    U_{\text{all}}(\sigma,\mu) &= \left[ \min\big(q^b,\tfrac{a\ell_T(\sigma)}{k}\big)\right]^{\ell_T(\sigma)} \ge 
    \left[ \min\big(q^b,a(1-\eta)^2\big)\right]^k =C_7^k, \\
    \end{align*}
    for some constants $C_6$ and $C_7$.
   Combining this with \eqref{eq:sizeO}, we get
   \[\sum_{\sigma \in \Omega_k^\eta} U_{\text{all}}(\sigma,\lambda) U_{\text{all}}(\sigma,\mu) \ge
    (C_5)^k\, k^{(1-\eta)k} \big(C_6 \tfrac{k}{n}\big)^k C_7^k.\]
    Call RHS the right-hand side of the previous display. We have
    \[\log(\textrm{RHS})=(2k -\eta)\log k - k \log n +k\log(C_5\, C_6\, C_7)=k \big(\log \tfrac{k^2}{n}+\O(1)\big).\]
    This completes the proof of the proposition.
\end{proof}

\section{Numerical evidence}
\label{sec:numerics}
In this section, we give numerical evidence to support Conjecture \ref{conj} and the fact that the bound of Theorem \ref{th: k medium} is sharp.

In \cite{KimOh}, Kim and Oh proved an exact product formula for the number of standard Young tableaux of shape $\la/\mu$ represented in Figure 1.
This formula was then independently rediscovered by Morales, Pak and Panova \cite{MPP3},
under the following form (which is different from Kim-Oh's original formulation).

\begin{figure}[ht]
  \centering
    \includegraphics[width=0.25\textwidth]{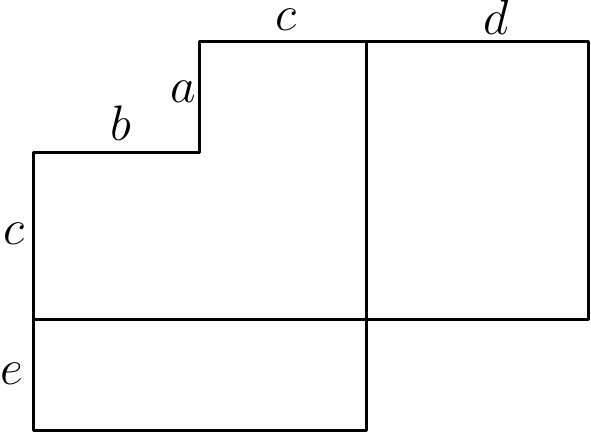}
      \caption{The skew shape $\la/\mu$}
\end{figure}

\begin{theorem}[Kim-Oh, Morales-Pak-Panova]
\label{th:KimOh}
For all $a,b,c,d,e \in \N$, the number $f^{\la/\mu}$ of standard Young tableaux of shape $\la/\mu$ given in Figure 1 is equal to
\begin{align*}
f^{\la/\mu} &= \big((a+c+e)(b+c+d)-ab-ed\big)!\\
& \times \frac{\Phi(a)\Phi(b)\Phi(c)\Phi(d)\Phi(e)\Phi(a+b+c)\Phi(c+d+e)\Phi(a+b+c+d+e)}{\Phi(a+b)\Phi(d+e)\Phi(a+c+d)\Phi(b+c+e)\Phi(a+b+2c+d+e)},
\end{align*}
where $\Phi$ is the superfactorial defined as
$$\Phi(n) := 1! \times 2! \times \cdots \times (n-1)! .$$
\end{theorem}

For such shapes, the numbers $f^\la$ and $f^\mu$ can also easily be written as quotients of superfactorials by using the hook-length formula.
Asymptotic expansions of logarithms of factorials (and thus of superfactorials) are known at any order,
so it is possible in principle to compute asymptotic expansions of $\log A_{\la,\mu}$ at any order for such shapes.
The computation is however very cumbersome in practice and best done with a computer by specializing $a,b,c,d,e$
to some functions of a single parameter $n$.

We summarise some examples of the asymptotic behaviour of $\log A_{\la/\mu}$ in the tables below.
Note that here $|\lambda|$ is $\Theta(n)$ and not exactly $n$ as before.
It makes the presentation simpler and hopefully will not create any confusion for the reader, 
since we are only interested in the order of magnitude of $A_{\la,\mu}$ and do not care about constants. 

\bigskip

\begin{center}
\label{table:small k}
$\begin{array}{|c|c|c|c|c|c||c|c|}
  \hline
 a & b & c & d & e & \log A_{\la/\mu} & |\mu| & |\mu|^{\frac{3}{2}}|\la|^{-\frac{1}{2}} \\
  \hline
   n^{\frac{1}{17}} & 2n^{\frac{1}{17}} & n^{\frac{1}{2}} & 2n^{\frac{1}{2}} & n^{\frac{1}{2}} & \frac{1}{4} n^{-\frac{11}{34}} + \O \left( n^{-\frac{13}{17}}\right) & 2n^{\frac{2}{17}} & \Theta \left(n^{-\frac{11}{34}} \right) \\

2n^{\frac{1}{13}} & n^{\frac{1}{13}} & n^{\frac{1}{2}} & n^{\frac{1}{2}} & n^{\frac{1}{3}} & -\frac{1}{2} n^{-\frac{7}{26}} + \O \left( n^{-\frac{17}{39}}\right) & 2n^{\frac{2}{13}} & \Theta \left(n^{-\frac{7}{26}} \right) \\

2n^{\frac{1}{10}} & n^{\frac{1}{10}} & n^{\frac{1}{2}} & n^{\frac{1}{3}} & n^{\frac{1}{2}} & \frac{1}{2} n^{-\frac{1}{5}} + \O \left( n^{-\frac{11}{30}}\right) & 2n^{\frac{1}{5}} & \Theta \left(n^{-\frac{1}{5}} \right) \\

n^{\frac{1}{6}} & 2n^{\frac{1}{6}} & n^{\frac{1}{2}} & n^{\frac{1}{2}} & n^{\frac{1}{3}} & \frac{1}{2} + \O \left( n^{-\frac{1}{6}}\right) & 2n^{\frac{1}{3}} & \Theta \left(1\right) \\

n^{\frac{1}{6}} & n^{\frac{1}{6}} & n^{\frac{1}{2}} & n^{\frac{1}{2}} & n^{\frac{1}{2}} & \frac{2}{27} n^{-\frac{1}{3}} + \O \left( n^{-\frac{2}{3}}\right) & n^{\frac{1}{3}} & \Theta \left(1\right) \\

n^{\frac{1}{5}} & 3n^{\frac{1}{5}} & n^{\frac{1}{2}} & n^{\frac{1}{2}} & n^{\frac{1}{3}} & \frac{3}{2} n^{\frac{1}{10}} + \O \left( n^{-\frac{1}{15}}\right) & 3n^{\frac{2}{5}} & \Theta \left(n^{\frac{1}{10}} \right) \\
  \hline
\end{array}$
\vspace{2pt}
    \captionof{table}{Asymptotics for $ \log A_{\la/\mu}$ when $|\mu| = o(|\la|^{\frac{1}{2}})$}
\end{center}

Table 1 gives evidence that the bound in Theorem \ref{th: k medium} is sharp.
Indeed, this bound is reached for skew diagrams with $|\mu| = \Theta(|\la|^\alpha)$,
for various values of $\alpha$ between $0$ and $1/2$.
\medskip

\begin{center}
\label{table:big k}
$\begin{array}{|c|c|c|c|c|c||c|c|}
  \hline
 a & b & c & d & e & \log A_{\la/\mu} & |\mu| & |\mu|^{\frac{3}{2}}|\la|^{-\frac{1}{2}} \\
  \hline
n^{\frac{1}{4}} & 3n^{\frac{1}{4}} & n^{\frac{1}{2}} & n^{\frac{1}{2}} & n^{\frac{1}{3}} & \frac{3}{2} n^{\frac{1}{4}} + \O \left( n^{\frac{1}{12}}\right) & 3n^{\frac{1}{2}} & \Theta \left(n^{\frac{1}{4}} \right) \\

2000n^{\frac{1}{4}} & 400n^{\frac{1}{4}} & n^{\frac{1}{2}} & n^{\frac{1}{2}} & 2n^{\frac{1}{2}} & 160000000 n^{\frac{1}{4}} + \O \left(1\right) & 800000n^{\frac{1}{2}} & \Theta \left(n^{\frac{1}{4}} \right) \\

n^{\frac{1}{3}} & 2n^{\frac{1}{3}} & n^{\frac{1}{2}} & 2n^{\frac{1}{2}} & n^{\frac{1}{2}} & \frac{1}{4} n^{\frac{1}{2}} + \O \left( n^{\frac{1}{3}}\right) & 2n^{\frac{2}{3}} & \Theta \left(n^{\frac{1}{2}} \right) \\

n^{\frac{2}{5}} & 2n^{\frac{2}{5}}  & n^{\frac{1}{2}} & 2n^{\frac{1}{2}} & n^{\frac{1}{2}} & \frac{1}{4} n^{\frac{7}{10}} + \O \left( n^{\frac{3}{5}}\right) &  2n^{\frac{4}{5}} & \Theta \left(n^{\frac{7}{10}} \right) \\

n^{\frac{2}{5}} & 2n^{\frac{2}{5}}  & n^{\frac{1}{2}} & n^{\frac{1}{2}} & 2n^{\frac{1}{2}} & -\frac{1}{4} n^{\frac{7}{10}} + \O \left( n^{\frac{3}{5}}\right) &  2n^{\frac{4}{5}} & \Theta \left(n^{\frac{7}{10}} \right) \\

n^{\frac{2}{5}} & n^{\frac{2}{5}}  & n^{\frac{1}{2}} & n^{\frac{1}{2}} & n^{\frac{1}{2}} & \frac{2}{27} n^{\frac{2}{5}} + \O \left( n^{\frac{2}{5}}\right) &  n^{\frac{4}{5}} & \Theta \left(n^{\frac{7}{10}} \right) \\

n^{\frac{1}{2}} & 2n^{\frac{1}{2}} & 3n^{\frac{1}{2}} & n^{\frac{1}{2}} & n^{\frac{1}{2}} & 0.56746 n + \O \left(1\right) & 2n & \Theta \left(n \right) \\
  \hline
\end{array}$
\vspace{2pt}
    \captionof{table}{Asymptotics for $ \log A_{\la/\mu}$ when $|\mu| \geq C |\la|^{\frac{1}{2}}$}
\end{center}

In the cases represented in Table 2 (and in all other cases that we have investigated), 
we observe that $\log A_{\la/\mu}$ is at most of order $|\mu|^{\frac{3}{2}}|\la|^{-\frac{1}{2}}$. Here are some further comments:
\begin{itemize}
  \item 
The first and second lines focus on the case where $|\mu| \sim C |\la|^{1/2}$ for some constant $C$.
Comparing \cref{th: k medium} and \cref{th: k big} could let us think that the order of magnitude of $ \log A_{\la/\mu}$
depends on the constant $C$.
The first and second lines of the table do not exhibit such a dependence.
\item The fifth line of the table gives an example,
  where $\log A_{\la/\mu}$ behaves as 
  $\Theta(|\mu|^{\frac{3}{2}}|\la|^{-\frac{1}{2}})$, with a negative constant.
  In this case $A_{\la/\mu}$ reaches the lower bound given in \cref{conj}.
  Intermediate situations, where $\log A_{\la/\mu}$ is negligible compared
  to $|\mu|^{\frac{3}{2}}|\la|^{-\frac{1}{2}}$, do also occur,
  see the second last line of the table.
\end{itemize}

All the results are evidence for Conjecture \ref{conj}.
However, as argued in the previous section,
proving this conjecture would require new techniques or ideas.

\section*{Acknowledgements}
The authors thank Igor Pak for stimulating conversations.
We would also like to thank Christian Krattenthaler for pointing out reference \cite{muller07fuchsian}.
Finally, we thank an anonymous referee for useful suggestions
to improve the paper's presentation.

\bibliographystyle{siam}
\bibliography{biblio}

\end{document}